\documentclass{amsart}
\usepackage{amssymb}
\usepackage{amsfonts}

\newtheorem{theorem}{Theorem}
\theoremstyle{plain}

\newtheorem{definition}{Definition}

\newtheorem{lemma}{Lemma}

\newtheorem{remark}{Remark}

\numberwithin{equation}{section}
\input{tcilatex}

\begin{document}
\title[Co-ordinated convex functions ]{New some Hadamard's type inequalities
for co-ordinated convex functions }
\author{Mehmet Zeki Sar\i kaya$^{\star \clubsuit }$}
\address{$^{\clubsuit }$Department of Mathematics,Faculty of Science and
Arts, D\"{u}zce University, D\"{u}zce, Turkey}
\email{sarikayamz@gmail.com}
\thanks{$^{\star }$corresponding author}
\author{Erhan. SET$^{\blacksquare }$}
\address{$^{\blacksquare }$Atat\"{u}rk University, K.K. Education Faculty,
Department of Mathematics, 25240, Campus, Erzurum, Turkey}
\email{erhanset@yahoo.com}
\author{M. Emin Ozdemir$^{\blacklozenge }$}
\address{$^{\blacklozenge }$Graduate School of Natural and Applied Sciences,
A\u{g}r\i\ \.{I}brahim \c{C}e\c{c}en University, A\u{g}r\i , Turkey}
\email{emos@atauni.edu.tr}
\author{Sever S. Dragomir$^{\blacktriangledown }$}
\address{$^{\blacktriangledown }$Research Group in Mathematical Inequalities
\& Applications\\
School of Engineering \& Science\\
Victoria University, PO Box 14428\\
Melbourne City, MC 8001, Australia.}
\email{sever.dragomir@vu.edu.au}
\urladdr{http://rgmia.vu.edu.au/dragomir}
\date{}
\subjclass[2000]{ 26A51, 26D15.}
\keywords{ convex function, co-ordinated convex mapping, Hermite-Hadamard
inequality}

\begin{abstract}
In this paper, we establish new some Hermite-Hadamard's type inequalities of
convex functions of $2-$variables on the co-ordinates.
\end{abstract}

\maketitle

\section{Introduction}

Let $f:I\subseteq \mathbb{R\rightarrow R}$ be a convex mapping defined on
the interval $I$ of real numbers and $a,b\in I$, with $a<b.$ the following
double inequality is well known in the literature as the Hermite-Hadamard
inequality:%
\begin{equation*}
f\left( \frac{a+b}{2}\right) \leq \frac{1}{b-a}\int_{a}^{b}f\left( x\right)
dx\leq \frac{f\left( a\right) +f\left( b\right) }{2}.
\end{equation*}

Let us now consider a bidemensional interval $\Delta =:\left[ a,b\right]
\times \left[ c,d\right] $ in $\mathbb{R}^{2}$ with $a<b$ and $c<d$. A
mapping $f:\Delta \rightarrow \mathbb{R}$ is said to be convex on $\Delta $
if the following inequality:%
\begin{equation*}
f(tx+\left( 1-t\right) z,ty+\left( 1-t\right) w)\leq tf\left( x,y\right)
+\left( 1-t\right) f\left( z,w\right)
\end{equation*}%
holds, for all $\left( x,y\right) ,\left( z,w\right) \in \Delta $ and $t\in %
\left[ 0,1\right] .$A function $f:\Delta \rightarrow \mathbb{R}$ is said to
be on the co-ordinates on $\Delta $ if the partial mappings $f_{y}:\left[ a,b%
\right] \rightarrow \mathbb{R},$ \ $f_{y}\left( u\right) =f\left( u,y\right) 
$ and $f_{x}:\left[ c,d\right] \rightarrow \mathbb{R},$ \ $f_{x}\left(
v\right) =f\left( x,v\right) $ are convex where defined for all $x\in \left[
a,b\right] $ and $y\in \left[ c,d\right] \ $(see \cite{D}).

A formal definition for co-ordinated convex function may be stated as
follows:

\begin{definition}
A function $f:\Delta \rightarrow \mathbb{R}$ will be called co-ordinated
canvex on $\Delta $, for all $t,s\in \lbrack 0,1]$ and $(x,y),(u,v)\in
\Delta ,$if the following inequality holds:%
\begin{eqnarray*}
&&f(tx+\left( 1-t\right) y,su+\left( 1-s\right) v) \\
&\leq &tsf(x,u)+s(1-t)f(y,u)+t(1-s)f(x,v)+(1-t)(1-s)f(y,v).
\end{eqnarray*}
\end{definition}

Clearly, every convex function is co-ordinated convex. Furthermore, there
exist co-ordinated convex function which is not convex, (see, \cite{D}). For
several recent results concerning Hermite-Hadamard's inequality for some
convex function on the co-ordinates on a rectangle from the plane $\mathbb{R}%
^{2},$ we refer the reader to (\cite{AD1}-\cite{MEO}).

Also, in \cite{D}, Dragomir establish the following similar inequality of
Hadamard's type for co-ordinated convex mapping on a rectangle from the
plane $\mathbb{R}^{2}.$

\begin{theorem}
\label{t.1.3} Suppose that $f:\Delta \rightarrow \mathbb{R}$ is co-ordinated
convex on $\Delta .\ $Then one has the inequalities:%
\begin{eqnarray}
&&f\left( \dfrac{a+b}{2},\dfrac{c+d}{2}\right)  \label{E8} \\
&\leq &\dfrac{1}{2}\left[ \dfrac{1}{b-a}\dint_{a}^{b}f\left( x,\dfrac{c+d}{2}%
\right) dx+\dfrac{1}{d-c}\dint_{c}^{d}f\left( \dfrac{a+b}{2},y\right) dy%
\right]  \notag \\
&\leq &\dfrac{1}{\left( b-a\right) \left( d-c\right) }\dint_{a}^{b}%
\dint_{c}^{d}f\left( x,y\right) dydx  \notag \\
&\leq &\dfrac{1}{4}\left[ \dfrac{1}{b-a}\dint_{a}^{b}f\left( x,c\right) dx+%
\dfrac{1}{b-a}\dint_{a}^{b}f\left( x,d\right) dx\right.  \notag \\
&&\left. +\dfrac{1}{d-c}\dint_{c}^{d}f\left( a,y\right) dy+\dfrac{1}{d-c}%
\dint_{c}^{d}f\left( b,y\right) dy\right]  \notag \\
&\leq &\dfrac{f\left( a,c\right) +f\left( a,d\right) +f\left( b,c\right)
+f\left( b,d\right) }{4}.  \notag
\end{eqnarray}%
The above inequalities are sharp.
\end{theorem}

The main purpose of this paper is to establish new Hadamard-type
inequalities of convex functions of $2$-variables on the co-ordinates.

\section{Inequalities for co-ordinated convex functions}

\begin{lemma}
\label{z} Let $f:\Delta \subset \mathbb{R}^{2}\rightarrow \mathbb{R}^{2}$ be
a partial differentiable mapping on $\Delta :=\left[ a,b\right] \times \left[
c,d\right] $ in $\mathbb{R}^{2}$ with $a<b$ and $c<d$. If\ $\dfrac{\partial
^{2}f}{\partial t\partial s}\in L(\Delta )$, then the following equality
holds:%
\begin{eqnarray}
&&\dfrac{f\left( a,c\right) +f\left( a,d\right) +f\left( b,c\right) +f\left(
b,d\right) }{4}+\dfrac{1}{\left( b-a\right) \left( d-c\right) }%
\dint_{a}^{b}\dint_{c}^{d}f\left( x,y\right) dydx  \label{1} \\
&&-\dfrac{1}{2}\left[ \dfrac{1}{b-a}\dint_{a}^{b}\left[ f\left( x,c\right)
+f\left( x,d\right) \right] dx+\dfrac{1}{d-c}\dint_{c}^{d}\left[ f\left(
a,y\right) +f\left( b,y\right) \right] dy\right]  \notag \\
&=&\dfrac{\left( b-a\right) \left( d-c\right) }{4}\dint_{0}^{1}%
\dint_{0}^{1}(1-2t)(1-2s)\dfrac{\partial ^{2}f}{\partial t\partial s}\left(
ta+(1-t)b,sc+(1-s)d\right) dtds.  \notag
\end{eqnarray}
\end{lemma}

\begin{proof}
By integration by parts, we get%
\begin{eqnarray}
&&\dint_{0}^{1}\dint_{0}^{1}(1-2s)(1-2t)\dfrac{\partial ^{2}f}{\partial
t\partial s}\left( ta+(1-t)b,sc+(1-s)d\right) dtds  \label{2} \\
&=&\dint_{0}^{1}(1-2s)\left\{ (1-2t)\left. \dfrac{1}{a-b}\dfrac{\partial f}{%
\partial s}\left( ta+(1-t)b,sc+(1-s)d\right) \right\vert _{0}^{1}\right. 
\notag \\
&&\ +\left. \dfrac{2}{a-b}\dint_{0}^{1}\dfrac{\partial f}{\partial s}\left(
ta+(1-t)b,sc+(1-s)d\right) dt\right\} ds  \notag \\
&=&\dint_{0}^{1}(1-2s)\left\{ -\dfrac{1}{a-b}\dfrac{\partial f}{\partial s}%
\left( a,sc+(1-s)d\right) -\dfrac{1}{a-b}\dfrac{\partial f}{\partial s}%
\left( b,sc+(1-s)d\right) \right.  \notag \\
&&+\left. \dfrac{2}{a-b}\dint_{0}^{1}\dfrac{\partial f}{\partial s}\left(
ta+(1-t)b,sc+(1-s)d\right) dt\right\} ds  \notag
\end{eqnarray}%
\begin{eqnarray}
&=&\dfrac{1}{b-a}\left\{ \dint_{0}^{1}(1-2s)\left( \dfrac{\partial f}{%
\partial s}\left( a,sc+(1-s)d\right) +\dfrac{\partial f}{\partial s}\left(
b,sc+(1-s)d\right) \right) ds\right.  \notag \\
&&\left. -2\dint_{0}^{1}\dint_{0}^{1}(1-2s)\dfrac{\partial f}{\partial s}%
\left( ta+(1-t)b,sc+(1-s)d\right) dtds\right\} .  \notag
\end{eqnarray}%
Thus, again by integration by parts in the right hand side of (\ref{2}), it
follows that%
\begin{eqnarray}
&&\dint_{0}^{1}(1-2s)\left( \dfrac{\partial f}{\partial s}\left(
a,sc+(1-s)d\right) +\dfrac{\partial f}{\partial s}\left( b,sc+(1-s)d\right)
\right) ds  \label{3} \\
&&-2\dint_{0}^{1}\dint_{0}^{1}(1-2s)\dfrac{\partial f}{\partial s}\left(
ta+(1-t)b,sc+(1-s)d\right) dtds  \notag \\
&=&(1-2s)\left. \dfrac{\left( f\left( a,sc+(1-s)d\right) +f\left(
b,sc+(1-s)d\right) \right) }{c-d}\right\vert _{0}^{1}  \notag \\
&&+\dfrac{2}{c-d}\dint_{0}^{1}\left( f\left( a,sc+(1-s)d\right) +f\left(
b,sc+(1-s)d\right) \right) ds  \notag \\
&&-2\dint_{0}^{1}\left\{ (1-2s)\left. \dfrac{f\left(
ta+(1-t)b,sc+(1-s)d\right) }{c-d}\right\vert _{0}^{1}\right.  \notag \\
&&\ +\left. \dfrac{2}{c-d}\dint_{0}^{1}f\left( ta+(1-t)b,sc+(1-s)d\right)
ds\right\} dt  \notag
\end{eqnarray}%
\begin{eqnarray*}
&=&-\dfrac{f\left( a,c\right) +f\left( b,c\right) }{c-d}-\dfrac{f\left(
a,d\right) +f\left( b,d\right) }{c-d} \\
&&+\dfrac{2}{c-d}\dint_{0}^{1}\left( f\left( a,sc+(1-s)d\right) +f\left(
b,sc+(1-s)d\right) \right) ds \\
&&-2\dint_{0}^{1}\left\{ -\dfrac{f\left( ta+(1-t)b,c\right) }{c-d}-\dfrac{%
f\left( ta+(1-t)b,d\right) }{c-d}\right. \\
&&\ \left. +\dfrac{2}{c-d}\dint_{0}^{1}f\left( ta+(1-t)b,sc+(1-s)d\right)
ds\right\} dt \\
&=&\dfrac{f\left( a,c\right) +f\left( a,d\right) +f\left( b,c\right)
+f\left( b,d\right) }{\left( d-c\right) } \\
&&+\dfrac{4}{\left( d-c\right) }\dint_{0}^{1}\dint_{0}^{1}f\left(
ta+(1-t)b,sc+(1-s)d\right) dsdt \\
&&-\dfrac{2}{\left( d-c\right) }\left\{ \dint_{0}^{1}\left( f\left(
a,sc+(1-s)d\right) +f\left( b,sc+(1-s)d\right) \right) ds\right. \\
&&+\left. \dint_{0}^{1}\left( f\left( ta+(1-t)b,c\right) +f\left(
ta+(1-t)b,d\right) \right) dt\right\} .
\end{eqnarray*}%
Writing (\ref{3}) in (\ref{2}), using the change of the variable $%
x=ta+(1-t)b $ and $y=sc+(1-s)d$ for $t,s\in \lbrack 0,1]^{2}$, and
multiplying the both sides by $\frac{\left( b-a\right) \left( d-c\right) }{4}%
,$ we obtain (\ref{1}), which completes the proof.
\end{proof}

\begin{theorem}
\label{t.2.1} Let $f:\Delta \subset \mathbb{R}^{2}\rightarrow \mathbb{R}^{2}$
be a partial differentiable mapping on $\Delta :=\left[ a,b\right] \times %
\left[ c,d\right] $ in $\mathbb{R}^{2}$ with $a<b$ and $c<d$. If\ $%
\left\vert \dfrac{\partial ^{2}f}{\partial t\partial s}\right\vert $ is a
convex function on the co-ordinates on $\Delta ,$ then one has the
inequalities:%
\begin{eqnarray}
&&\left\vert \dfrac{f\left( a,c\right) +f\left( a,d\right) +f\left(
b,c\right) +f\left( b,d\right) }{4}\right.  \label{S0} \\
&&\left. +\dfrac{1}{\left( b-a\right) \left( d-c\right) }\dint_{a}^{b}%
\dint_{c}^{d}f\left( x,y\right) dydx-A\right\vert  \notag \\
&\leq &\dfrac{\left( b-a\right) \left( d-c\right) }{16}\left( \dfrac{%
\left\vert \frac{\partial ^{2}f}{\partial s\partial t}(a,c)\right\vert
+\left\vert \frac{\partial ^{2}f}{\partial s\partial t}(a,d)\right\vert
+\left\vert \frac{\partial ^{2}f}{\partial s\partial t}(b,c)\right\vert
+\left\vert \frac{\partial ^{2}f}{\partial s\partial t}(b,d)\right\vert }{4}%
\right)  \notag
\end{eqnarray}%
where%
\begin{equation*}
A=\dfrac{1}{2}\left[ \dfrac{1}{b-a}\dint_{a}^{b}\left[ f\left( x,c\right)
+f\left( x,d\right) \right] dx+\dfrac{1}{d-c}\dint_{c}^{d}\left[ f\left(
a,y\right) +f\left( b,y\right) \right] dy\right] .
\end{equation*}
\end{theorem}

\begin{proof}
From Lemma \ref{z}, we have%
\begin{eqnarray*}
&&\left\vert \dfrac{f\left( a,c\right) +f\left( a,d\right) +f\left(
b,c\right) +f\left( b,d\right) }{4}\right. \\
&&\left. +\dfrac{1}{\left( b-a\right) \left( d-c\right) }\dint_{a}^{b}%
\dint_{c}^{d}f\left( x,y\right) dydx-A\right\vert \\
&\leq &\dfrac{\left( b-a\right) \left( d-c\right) }{4} \\
&&\times \dint_{0}^{1}\dint_{0}^{1}\left\vert (1-2t)(1-2s)\right\vert
\left\vert \dfrac{\partial ^{2}f}{\partial t\partial s}\left(
ta+(1-t)b,sc+(1-s)d\right) \right\vert dtds.
\end{eqnarray*}%
Since $f:\Delta \rightarrow \mathbb{R}$ is co-ordinated convex on $\Delta $,
then one has:%
\begin{eqnarray*}
&&\left\vert \dfrac{f\left( a,c\right) +f\left( a,d\right) +f\left(
b,c\right) +f\left( b,d\right) }{4}\right. \\
&&\left. +\dfrac{1}{\left( b-a\right) \left( d-c\right) }\dint_{a}^{b}%
\dint_{c}^{d}f\left( x,y\right) dydx-A\right\vert \\
&\leq &\dfrac{\left( b-a\right) \left( d-c\right) }{4} \\
&&\times \dint_{0}^{1}\left[ \dint_{0}^{1}\left\vert (1-2t)(1-2s)\right\vert
\left\{ t\left\vert \dfrac{\partial ^{2}f}{\partial t\partial s}\left(
a,sc+(1-s)d\right) \right\vert \right. \right. \\
&&\left. \left. +(1-t)\left\vert \dfrac{\partial ^{2}f}{\partial t\partial s}%
\left( b,sc+(1-s)d\right) \right\vert \right\} dt\right] ds.
\end{eqnarray*}

Firstly, by calculating the integral in above inequality, we have%
\begin{eqnarray*}
&&\dint_{0}^{1}\left\vert 1-2t\right\vert \left\{ t\left\vert \dfrac{%
\partial ^{2}f}{\partial t\partial s}\left( a,sc+(1-s)d\right) \right\vert
+(1-t)\left\vert \dfrac{\partial ^{2}f}{\partial t\partial s}\left(
b,sc+(1-s)d\right) \right\vert \right\} dt \\
&=&\dint_{0}^{\frac{1}{2}}(1-2t)\left\{ t\left\vert \dfrac{\partial ^{2}f}{%
\partial t\partial s}\left( a,sc+(1-s)d\right) \right\vert +(1-t)\left\vert 
\dfrac{\partial ^{2}f}{\partial t\partial s}\left( b,sc+(1-s)d\right)
\right\vert \right\} dt \\
&&+\dint_{\frac{1}{2}}^{1}(2t-1)\left\{ t\left\vert \dfrac{\partial ^{2}f}{%
\partial t\partial s}\left( a,sc+(1-s)d\right) \right\vert +(1-t)\left\vert 
\dfrac{\partial ^{2}f}{\partial t\partial s}\left( b,sc+(1-s)d\right)
\right\vert \right\} dt
\end{eqnarray*}%
\begin{equation*}
=\dfrac{1}{4}\left( \left\vert \dfrac{\partial ^{2}f}{\partial t\partial s}%
\left( a,sc+(1-s)d\right) \right\vert +\left\vert \dfrac{\partial ^{2}f}{%
\partial t\partial s}\left( b,sc+(1-s)d\right) \right\vert \right) .
\end{equation*}

Thus, we obtain%
\begin{eqnarray}
&&\left\vert \dfrac{f\left( a,c\right) +f\left( a,d\right) +f\left(
b,c\right) +f\left( b,d\right) }{4}\right.  \label{S1} \\
&&\left. +\dfrac{1}{\left( b-a\right) \left( d-c\right) }\dint_{a}^{b}%
\dint_{c}^{d}f\left( x,y\right) dydx-A\right\vert  \notag \\
&\leq &\dfrac{\left( b-a\right) \left( d-c\right) }{16}  \notag \\
&&\times \dint_{0}^{1}\left\vert 1-2s\right\vert \left\{ \left\vert \dfrac{%
\partial ^{2}f}{\partial t\partial s}\left( a,sc+(1-s)d\right) \right\vert
+\left\vert \dfrac{\partial ^{2}f}{\partial t\partial s}\left(
b,sc+(1-s)d\right) \right\vert \right\} ds.  \notag
\end{eqnarray}%
A similar way for other integral, since $f:\Delta \rightarrow \mathbb{R}$ is
co-ordinated convex on $\Delta ,$ we get%
\begin{eqnarray}
&&\dint_{0}^{1}\left\vert 1-2s\right\vert \left\{ \left\vert \dfrac{\partial
^{2}f}{\partial t\partial s}\left( a,sc+(1-s)d\right) \right\vert
+\left\vert \dfrac{\partial ^{2}f}{\partial t\partial s}\left(
b,sc+(1-s)d\right) \right\vert \right\} ds  \label{S2} \\
&=&\dint_{0}^{\frac{1}{2}}(1-2s)\left\{ s\left\vert \dfrac{\partial ^{2}f}{%
\partial t\partial s}\left( a,c\right) \right\vert +(1-s)\left\vert \dfrac{%
\partial ^{2}f}{\partial t\partial s}\left( a,d\right) \right\vert \right\}
ds  \notag \\
&&+\dint_{0}^{\frac{1}{2}}(1-2s)\left\{ s\left\vert \dfrac{\partial ^{2}f}{%
\partial t\partial s}\left( b,c\right) \right\vert +(1-s)\left\vert \dfrac{%
\partial ^{2}f}{\partial t\partial s}\left( b,d\right) \right\vert \right\}
ds  \notag \\
&=&\dint_{\frac{1}{2}}^{1}(2s-1)\left\{ s\left\vert \dfrac{\partial ^{2}f}{%
\partial t\partial s}\left( a,c\right) \right\vert +(1-s)\left\vert \dfrac{%
\partial ^{2}f}{\partial t\partial s}\left( a,d\right) \right\vert \right\}
ds  \notag \\
&&\ +\dint_{\frac{1}{2}}^{1}(2s-1)\left\{ s\left\vert \dfrac{\partial ^{2}f}{%
\partial t\partial s}\left( b,c\right) \right\vert +(1-s)\left\vert \dfrac{%
\partial ^{2}f}{\partial t\partial s}\left( b,d\right) \right\vert \right\}
ds  \notag \\
&=&\dfrac{\left\vert \dfrac{\partial ^{2}f}{\partial t\partial s}\left(
a,c\right) \right\vert +\left\vert \dfrac{\partial ^{2}f}{\partial t\partial
s}\left( a,d\right) \right\vert +\left\vert \dfrac{\partial ^{2}f}{\partial
t\partial s}\left( b,c\right) \right\vert +\left\vert \dfrac{\partial ^{2}f}{%
\partial t\partial s}\left( b,d\right) \right\vert }{4}.  \notag
\end{eqnarray}%
By the (\ref{S1}) and (\ref{S2}), we get the inequality (\ref{S0}).
\end{proof}

\begin{theorem}
\label{t.2.2} Let $f:\Delta \subset \mathbb{R}^{2}\rightarrow \mathbb{R}^{2}$
be a partial differentiable mapping on $\Delta :=\left[ a,b\right] \times %
\left[ c,d\right] $ in $\mathbb{R}^{2}$ with $a<b$ and $c<d$. If\ $\
\left\vert \dfrac{\partial ^{2}f}{\partial t\partial s}\right\vert ^{q},$ $%
q>1,$is a convex function on the co-ordinates on $\Delta ,$ then one has the
inequalities:%
\begin{eqnarray}
&&\left\vert \dfrac{f\left( a,c\right) +f\left( a,d\right) +f\left(
b,c\right) +f\left( b,d\right) }{4}\right. \\
&&\left. +\dfrac{1}{\left( b-a\right) \left( d-c\right) }\dint_{a}^{b}%
\dint_{c}^{d}f\left( x,y\right) dydx-A\right\vert  \notag \\
&\leq &\dfrac{\left( b-a\right) \left( d-c\right) }{4\left( p+1\right) ^{%
\frac{2}{p}}}  \notag \\
&&\times \left( \dfrac{\left\vert \frac{\partial ^{2}f}{\partial s\partial t}%
(a,c)\right\vert ^{q}+\left\vert \frac{\partial ^{2}f}{\partial s\partial t}%
(a,d)\right\vert ^{q}+\left\vert \frac{\partial ^{2}f}{\partial s\partial t}%
(b,c)\right\vert ^{q}+\left\vert \frac{\partial ^{2}f}{\partial s\partial t}%
(b,d)\right\vert ^{q}}{4}\right) ^{\frac{1}{q}}  \notag
\end{eqnarray}%
where%
\begin{equation*}
A=\dfrac{1}{2}\left[ \dfrac{1}{b-a}\dint_{a}^{b}\left[ f\left( x,c\right)
+f\left( x,d\right) \right] dx+\dfrac{1}{d-c}\dint_{c}^{d}\left[ f\left(
a,y\right) +f\left( b,y\right) \right] dy\right] .
\end{equation*}%
and $\frac{1}{p}+\frac{1}{q}=1.$
\end{theorem}

\begin{proof}
From Lemma \ref{z}, we have%
\begin{eqnarray*}
&&\left\vert \dfrac{f\left( a,c\right) +f\left( a,d\right) +f\left(
b,c\right) +f\left( b,d\right) }{4}\right. \\
&&\left. +\dfrac{1}{\left( b-a\right) \left( d-c\right) }\dint_{a}^{b}%
\dint_{c}^{d}f\left( x,y\right) dydx-A\right\vert \\
&\leq &\dfrac{\left( b-a\right) \left( d-c\right) }{4} \\
&&\times \dint_{0}^{1}\dint_{0}^{1}\left\vert (1-2t)(1-2s)\right\vert
\left\vert \dfrac{\partial ^{2}f}{\partial t\partial s}\left(
ta+(1-t)b,sc+(1-s)d\right) \right\vert dtds.
\end{eqnarray*}%
By using the well known H\"{o}lder inequality for double integrals, $%
f:\Delta \rightarrow \mathbb{R}$ is co-ordinated convex on $\Delta $, then
one has:%
\begin{eqnarray*}
&&\left\vert \dfrac{f\left( a,c\right) +f\left( a,d\right) +f\left(
b,c\right) +f\left( b,d\right) }{4}\right. \\
&&\left. +\dfrac{1}{\left( b-a\right) \left( d-c\right) }\dint_{a}^{b}%
\dint_{c}^{d}f\left( x,y\right) dydx-A\right\vert \\
&\leq &\dfrac{\left( b-a\right) \left( d-c\right) }{4}\left(
\dint_{0}^{1}\dint_{0}^{1}\left\vert (1-2t)(1-2s)\right\vert ^{p}dtds\right)
^{\frac{1}{p}} \\
&&\ \times \left( \dint_{0}^{1}\dint_{0}^{1}\left\vert \dfrac{\partial ^{2}f%
}{\partial t\partial s}\left( ta+(1-t)b,sc+(1-s)d\right) \right\vert
^{q}dtds\right) ^{\frac{1}{q}}.
\end{eqnarray*}%
Since $\left\vert \dfrac{\partial ^{2}f}{\partial t\partial s}\right\vert
^{q}$ is convex function on the co-ordinates on $\Delta ,$ we know that for $%
t\in \lbrack 0,1]$%
\begin{eqnarray*}
&&\left\vert \dfrac{\partial ^{2}f}{\partial t\partial s}%
(ta+(1-t)b,sc+(1-s)d)\right\vert ^{q} \\
&\leq &t\left\vert \dfrac{\partial ^{2}f}{\partial t\partial s}%
(a,sc+(1-s)d)\right\vert ^{q}+(1-t)\left\vert \dfrac{\partial ^{2}f}{%
\partial t\partial s}(b,sc+(1-s)d)\right\vert ^{q}
\end{eqnarray*}%
and%
\begin{eqnarray*}
&&\left\vert \dfrac{\partial ^{2}f}{\partial t\partial s}%
(ta+(1-t)b,sc+(1-s)d)\right\vert ^{q} \\
&\leq &ts\left\vert \dfrac{\partial ^{2}f}{\partial t\partial s}%
(a,c)\right\vert ^{q}+t(1-s)\left\vert \dfrac{\partial ^{2}f}{\partial
t\partial s}(a,d)\right\vert ^{q} \\
&&+(1-t)s\left\vert \dfrac{\partial ^{2}f}{\partial t\partial s}%
(b,c)\right\vert ^{q}+(1-t)(1-s)\left\vert \dfrac{\partial ^{2}f}{\partial
t\partial s}(b,c)\right\vert ^{q}
\end{eqnarray*}%
hence, it follows that%
\begin{eqnarray*}
&&\left\vert \dfrac{f\left( a,c\right) +f\left( a,d\right) +f\left(
b,c\right) +f\left( b,d\right) }{4}\right. \\
&&\left. +\dfrac{1}{\left( b-a\right) \left( d-c\right) }\dint_{a}^{b}%
\dint_{c}^{d}f\left( x,y\right) dydx-A\right\vert \\
\ &\leq &\dfrac{\left( b-a\right) \left( d-c\right) }{4\left( p+1\right) ^{%
\frac{2}{p}}} \\
&&\times \left( \dint_{0}^{1}\dint_{0}^{1}\left\{ ts\left\vert \dfrac{%
\partial ^{2}f}{\partial t\partial s}(a,c)\right\vert ^{q}+t(1-s)\left\vert 
\dfrac{\partial ^{2}f}{\partial t\partial s}(a,d)\right\vert ^{q}\right.
\right. \\
&&\ \left. \left. +(1-t)s\left\vert \dfrac{\partial ^{2}f}{\partial
t\partial s}(b,c)\right\vert ^{q}+(1-t)(1-s)\left\vert \dfrac{\partial ^{2}f%
}{\partial t\partial s}(b,d)\right\vert ^{q}\right\} dtds\right) ^{\frac{1}{q%
}} \\
&=&\dfrac{\left( b-a\right) \left( d-c\right) }{4\left( p+1\right) ^{\frac{2%
}{p}}} \\
&&\times \left( \dfrac{\left\vert \frac{\partial ^{2}f}{\partial s\partial t}%
(a,c)\right\vert ^{q}+\left\vert \frac{\partial ^{2}f}{\partial s\partial t}%
(a,d)\right\vert ^{q}+\left\vert \frac{\partial ^{2}f}{\partial s\partial t}%
(b,c)\right\vert ^{q}+\left\vert \frac{\partial ^{2}f}{\partial s\partial t}%
(b,d)\right\vert ^{q}}{4}\right) ^{\frac{1}{q}}.
\end{eqnarray*}
\end{proof}

\begin{theorem}
\label{t.2.3} Let $f:\Delta \subset \mathbb{R}^{2}\rightarrow \mathbb{R}^{2}$
be a partial differentiable mapping on $\Delta :=\left[ a,b\right] \times %
\left[ c,d\right] $ in $\mathbb{R}^{2}$ with $a<b$ and $c<d$. If\ $\
\left\vert \dfrac{\partial ^{2}f}{\partial t\partial s}\right\vert ^{q},$ $%
q\geq 1,$is a convex function on the co-ordinates on $\Delta ,$ then one has
the inequalities:%
\begin{eqnarray}
&&\left\vert \dfrac{f\left( a,c\right) +f\left( a,d\right) +f\left(
b,c\right) +f\left( b,d\right) }{4}\right.  \label{E1} \\
&&\left. +\dfrac{1}{\left( b-a\right) \left( d-c\right) }\dint_{a}^{b}%
\dint_{c}^{d}f\left( x,y\right) dydx-A\right\vert  \notag \\
&\leq &\dfrac{\left( b-a\right) \left( d-c\right) }{16}  \notag \\
&&\times \left( \dfrac{\left\vert \dfrac{\partial ^{2}f}{\partial t\partial s%
}(a,c)\right\vert ^{q}+\left\vert \dfrac{\partial ^{2}f}{\partial t\partial s%
}(a,d)\right\vert ^{q}+\left\vert \dfrac{\partial ^{2}f}{\partial t\partial s%
}(b,c)\right\vert ^{q}+\left\vert \dfrac{\partial ^{2}f}{\partial t\partial s%
}(b,d)\right\vert ^{q}}{4}\right) ^{\frac{1}{q}}  \notag
\end{eqnarray}%
where%
\begin{equation*}
A=\dfrac{1}{2}\left[ \dfrac{1}{b-a}\dint_{a}^{b}\left[ f\left( x,c\right)
+f\left( x,d\right) \right] dx+\dfrac{1}{d-c}\dint_{c}^{d}\left[ f\left(
a,y\right) +f\left( b,y\right) \right] dy\right] .
\end{equation*}
\end{theorem}

\begin{proof}
From Lemma \ref{z}, we have%
\begin{eqnarray*}
&&\left\vert \dfrac{f\left( a,c\right) +f\left( a,d\right) +f\left(
b,c\right) +f\left( b,d\right) }{4}\right. \\
&&\left. +\dfrac{1}{\left( b-a\right) \left( d-c\right) }\dint_{a}^{b}%
\dint_{c}^{d}f\left( x,y\right) dydx-A\right\vert \\
&\leq &\dfrac{\left( b-a\right) \left( d-c\right) }{4} \\
&&\times \dint_{0}^{1}\dint_{0}^{1}\left\vert (1-2t)(1-2s)\right\vert
\left\vert \dfrac{\partial ^{2}f}{\partial t\partial s}\left(
ta+(1-t)b,sc+(1-s)d\right) \right\vert dtds.
\end{eqnarray*}%
By using the well known power mean inequality for double integrals, $%
f:\Delta \rightarrow \mathbb{R}$ is co-ordinated convex on $\Delta $, then
one has:%
\begin{eqnarray*}
&&\left\vert \dfrac{f\left( a,c\right) +f\left( a,d\right) +f\left(
b,c\right) +f\left( b,d\right) }{4}\right. \\
&&\left. +\dfrac{1}{\left( b-a\right) \left( d-c\right) }\dint_{a}^{b}%
\dint_{c}^{d}f\left( x,y\right) dydx-A\right\vert \\
&\leq &\dfrac{\left( b-a\right) \left( d-c\right) }{4}\left(
\dint_{0}^{1}\dint_{0}^{1}\left\vert (1-2t)(1-2s)\right\vert dtds\right) ^{1-%
\frac{1}{q}} \\
&&\times \left( \dint_{0}^{1}\dint_{0}^{1}\left\vert (1-2t)(1-2s)\right\vert
\left\vert \dfrac{\partial ^{2}f}{\partial t\partial s}\left(
ta+(1-t)b,sc+(1-s)d\right) \right\vert ^{q}dtds\right) ^{\frac{1}{q}}.
\end{eqnarray*}%
Since $\left\vert \dfrac{\partial ^{2}f}{\partial t\partial s}\right\vert
^{q}$ is convex function on the co-ordinates on $\Delta ,$ we know that for $%
t\in \lbrack 0,1]$%
\begin{eqnarray*}
&&\left\vert \dfrac{\partial ^{2}f}{\partial t\partial s}%
(ta+(1-t)b,sc+(1-s)d)\right\vert ^{q} \\
&\leq &t\left\vert \dfrac{\partial ^{2}f}{\partial t\partial s}%
(a,sc+(1-s)d)\right\vert ^{q}+(1-t)\left\vert \dfrac{\partial ^{2}f}{%
\partial t\partial s}(b,sc+(1-s)d)\right\vert ^{q}
\end{eqnarray*}%
and%
\begin{eqnarray*}
&&\left\vert \dfrac{\partial ^{2}f}{\partial t\partial s}%
(ta+(1-t)b,sc+(1-s)d)\right\vert ^{q} \\
&\leq &ts\left\vert \dfrac{\partial ^{2}f}{\partial t\partial s}%
(a,c)\right\vert ^{q}+t(1-s)\left\vert \dfrac{\partial ^{2}f}{\partial
t\partial s}(a,d)\right\vert ^{q} \\
&&+(1-t)s\left\vert \dfrac{\partial ^{2}f}{\partial t\partial s}%
(b,c)\right\vert ^{q}+(1-t)(1-s)\left\vert \dfrac{\partial ^{2}f}{\partial
t\partial s}(b,c)\right\vert ^{q}
\end{eqnarray*}%
hence, it follows that%
\begin{eqnarray*}
&&\left\vert \dfrac{f\left( a,c\right) +f\left( a,d\right) +f\left(
b,c\right) +f\left( b,d\right) }{4}\right. \\
&&\left. +\dfrac{1}{\left( b-a\right) \left( d-c\right) }\dint_{a}^{b}%
\dint_{c}^{d}f\left( x,y\right) dydx-A\right\vert
\end{eqnarray*}%
\begin{eqnarray*}
&\leq &\dfrac{\left( b-a\right) \left( d-c\right) }{4}\left( \dfrac{1}{4}%
\right) ^{1-\frac{1}{q}} \\
&&\times \left( \dint_{0}^{1}\dint_{0}^{1}\left\vert (1-2t)(1-2s)\right\vert
\left\{ ts\left\vert \dfrac{\partial ^{2}f}{\partial t\partial s}%
(a,c)\right\vert ^{q}+t(1-s)\left\vert \dfrac{\partial ^{2}f}{\partial
t\partial s}(a,d)\right\vert ^{q}\right. \right. \\
&&\ \ \left. \left. +(1-t)s\left\vert \dfrac{\partial ^{2}f}{\partial
t\partial s}(b,c)\right\vert ^{q}+(1-t)(1-s)\left\vert \dfrac{\partial ^{2}f%
}{\partial t\partial s}(b,d)\right\vert ^{q}\right\} dtds\right) ^{\frac{1}{q%
}}.
\end{eqnarray*}%
Firstly, by calculating the integral in above inequality, we have%
\begin{eqnarray*}
&&\dint_{0}^{1}\left\vert 1-2t\right\vert \left( ts\left\vert \dfrac{%
\partial ^{2}f}{\partial t\partial s}(a,c)\right\vert ^{q}+t(1-s)\left\vert 
\dfrac{\partial ^{2}f}{\partial t\partial s}(a,d)\right\vert ^{q}\right. \\
&&\left. +(1-t)s\left\vert \dfrac{\partial ^{2}f}{\partial t\partial s}%
(b,c)\right\vert ^{q}+(1-t)(1-s)\left\vert \dfrac{\partial ^{2}f}{\partial
t\partial s}(b,d)\right\vert ^{q}\right) dt \\
&=&\dint_{0}^{\frac{1}{2}}\left( 1-2t\right) \left( ts\left\vert \dfrac{%
\partial ^{2}f}{\partial t\partial s}(a,c)\right\vert ^{q}+t(1-s)\left\vert 
\dfrac{\partial ^{2}f}{\partial t\partial s}(a,d)\right\vert ^{q}\right. \\
&&\left. +(1-t)s\left\vert \dfrac{\partial ^{2}f}{\partial t\partial s}%
(b,c)\right\vert ^{q}+(1-t)(1-s)\left\vert \dfrac{\partial ^{2}f}{\partial
t\partial s}(b,d)\right\vert ^{q}\right) dt \\
&&+\dint_{\frac{1}{2}}^{1}\left( 2t-1\right) \left( ts\left\vert \dfrac{%
\partial ^{2}f}{\partial t\partial s}(a,c)\right\vert ^{q}+t(1-s)\left\vert 
\dfrac{\partial ^{2}f}{\partial t\partial s}(a,d)\right\vert ^{q}\right. \\
&&\left. +(1-t)s\left\vert \dfrac{\partial ^{2}f}{\partial t\partial s}%
(b,c)\right\vert ^{q}+(1-t)(1-s)\left\vert \dfrac{\partial ^{2}f}{\partial
t\partial s}(b,d)\right\vert ^{q}\right) dt \\
&=&\dfrac{s\left\vert \dfrac{\partial ^{2}f}{\partial t\partial s}%
(a,c)\right\vert ^{q}}{24}+\dfrac{(1-s)\left\vert \dfrac{\partial ^{2}f}{%
\partial t\partial s}(a,d)\right\vert ^{q}}{24} \\
&&+\dfrac{5s\left\vert \dfrac{\partial ^{2}f}{\partial t\partial s}%
(b,c)\right\vert ^{q}}{24}+\dfrac{5(1-s)\left\vert \dfrac{\partial ^{2}f}{%
\partial t\partial s}(b,d)\right\vert ^{q}}{24} \\
&&+\dfrac{5s\left\vert \dfrac{\partial ^{2}f}{\partial t\partial s}%
(a,c)\right\vert ^{q}}{24}+\dfrac{5(1-s)\left\vert \dfrac{\partial ^{2}f}{%
\partial t\partial s}(a,d)\right\vert ^{q}}{24} \\
&&+\dfrac{s\left\vert \dfrac{\partial ^{2}f}{\partial t\partial s}%
(b,c)\right\vert ^{q}}{24}+\dfrac{(1-s)\left\vert \dfrac{\partial ^{2}f}{%
\partial t\partial s}(b,d)\right\vert ^{q}}{24} \\
&=&\dfrac{s\left\vert \dfrac{\partial ^{2}f}{\partial t\partial s}%
(a,c)\right\vert ^{q}+(1-s)\left\vert \dfrac{\partial ^{2}f}{\partial
t\partial s}(a,d)\right\vert ^{q}}{4} \\
&&+\frac{s\left\vert \dfrac{\partial ^{2}f}{\partial t\partial s}%
(b,c)\right\vert ^{q}+(1-s)\left\vert \dfrac{\partial ^{2}f}{\partial
t\partial s}(b,d)\right\vert ^{q}}{4}.
\end{eqnarray*}%
Thus, we obtain%
\begin{eqnarray}
&&\left\vert \dfrac{f\left( a,c\right) +f\left( a,d\right) +f\left(
b,c\right) +f\left( b,d\right) }{4}\right.  \label{E2} \\
&&\left. +\dfrac{1}{\left( b-a\right) \left( d-c\right) }\dint_{a}^{b}%
\dint_{c}^{d}f\left( x,y\right) dydx-A\right\vert  \notag \\
&\leq &\dfrac{\left( b-a\right) \left( d-c\right) }{16}\left[
\dint_{0}^{1}\left\vert 1-2s\right\vert \left( s\left\vert \dfrac{\partial
^{2}f}{\partial t\partial s}(a,c)\right\vert ^{q}\right. \right.  \notag \\
&&\left. \left. +(1-s)\left\vert \dfrac{\partial ^{2}f}{\partial t\partial s}%
(a,d)\right\vert ^{q}+s\left\vert \dfrac{\partial ^{2}f}{\partial t\partial s%
}(b,c)\right\vert ^{q}+(1-s)\left\vert \dfrac{\partial ^{2}f}{\partial
t\partial s}(b,d)\right\vert ^{q}\right) ds\right] ^{\frac{1}{q}}.  \notag
\end{eqnarray}%
A similar way for other integral, since $f:\Delta \rightarrow \mathbb{R}$ is
co-ordinated convex on $\Delta ,$ we get%
\begin{eqnarray}
&&\dint_{0}^{1}\left\vert 1-2s\right\vert \left( s\left\vert \dfrac{\partial
^{2}f}{\partial t\partial s}(a,c)\right\vert ^{q}+(1-s)\left\vert \dfrac{%
\partial ^{2}f}{\partial t\partial s}(a,d)\right\vert ^{q}\right.  \label{E3}
\\
&&\left. +s\left\vert \dfrac{\partial ^{2}f}{\partial t\partial s}%
(b,c)\right\vert ^{q}+(1-s)\left\vert \dfrac{\partial ^{2}f}{\partial
t\partial s}(b,d)\right\vert ^{q}\right) ds  \notag \\
&=&\dint_{0}^{\frac{1}{2}}\left( 1-2s\right) \left( s\left\vert \dfrac{%
\partial ^{2}f}{\partial t\partial s}(a,c)\right\vert ^{q}+(1-s)\left\vert 
\dfrac{\partial ^{2}f}{\partial t\partial s}(a,d)\right\vert ^{q}\right. 
\notag \\
&&\left. +s\left\vert \dfrac{\partial ^{2}f}{\partial t\partial s}%
(b,c)\right\vert ^{q}+(1-s)\left\vert \dfrac{\partial ^{2}f}{\partial
t\partial s}(b,d)\right\vert ^{q}\right) ds  \notag \\
&&+\dint_{\frac{1}{2}}^{1}\left( 2s-1\right) \left( s\left\vert \dfrac{%
\partial ^{2}f}{\partial t\partial s}(a,c)\right\vert ^{q}+(1-s)\left\vert 
\dfrac{\partial ^{2}f}{\partial t\partial s}(a,d)\right\vert ^{q}\right. 
\notag \\
&&\left. +s\left\vert \dfrac{\partial ^{2}f}{\partial t\partial s}%
(b,c)\right\vert ^{q}+(1-s)\left\vert \dfrac{\partial ^{2}f}{\partial
t\partial s}(b,d)\right\vert ^{q}\right) ds  \notag \\
&=&\dfrac{\left\vert \dfrac{\partial ^{2}f}{\partial t\partial s}%
(a,c)\right\vert ^{q}}{24}+\dfrac{5\left\vert \dfrac{\partial ^{2}f}{%
\partial t\partial s}(a,d)\right\vert ^{q}}{24}+\dfrac{\left\vert \dfrac{%
\partial ^{2}f}{\partial t\partial s}(b,c)\right\vert ^{q}}{24}+\dfrac{%
5\left\vert \dfrac{\partial ^{2}f}{\partial t\partial s}(b,d)\right\vert ^{q}%
}{24}  \notag \\
&&+\dfrac{5\left\vert \dfrac{\partial ^{2}f}{\partial t\partial s}%
(a,c)\right\vert ^{q}}{24}+\dfrac{\left\vert \dfrac{\partial ^{2}f}{\partial
t\partial s}(a,d)\right\vert ^{q}}{24}+\dfrac{5\left\vert \dfrac{\partial
^{2}f}{\partial t\partial s}(b,c)\right\vert ^{q}}{24}+\dfrac{\left\vert 
\dfrac{\partial ^{2}f}{\partial t\partial s}(b,d)\right\vert ^{q}}{24} 
\notag \\
&=&\dfrac{\left\vert \dfrac{\partial ^{2}f}{\partial t\partial s}%
(a,c)\right\vert ^{q}+\left\vert \dfrac{\partial ^{2}f}{\partial t\partial s}%
(a,d)\right\vert ^{q}+\left\vert \dfrac{\partial ^{2}f}{\partial t\partial s}%
(b,c)\right\vert ^{q}+\left\vert \dfrac{\partial ^{2}f}{\partial t\partial s}%
(b,d)\right\vert ^{q}}{4}.  \notag
\end{eqnarray}%
By the (\ref{E2}) and (\ref{E3}), we get the inequality (\ref{E1}).
\end{proof}

\begin{remark}
Since $\frac{1}{4}<\frac{1}{\left( p+1\right) ^{\frac{2}{p}}}<1,$ if $p>1,$
the estimation given in Theorem \ref{t.2.3} is better than the one given in
Theorem \ref{t.2.2}.
\end{remark}

\end{document}